\documentclass[12pt]{amsart}

\usepackage{color}

\usepackage{enumerate}

\newtheorem{theorem}{Theorem}[section]
\newtheorem{corollary}[theorem]{Corollary}
\newtheorem{lemma}[theorem]{Lemma}
\newtheorem{proposition}[theorem]{Proposition}

\theoremstyle{definition}
\newtheorem{definition}[theorem]{Definition}
\theoremstyle{remark}

\numberwithin{equation}{section}

\def\FF{\mathcal{F}}
\def\NN{\mathbb{N}}
\def\chain#1#2#3#4{(#1_{#3}, #2_{#3}),\ldots, (#1_{#4}, #2_{#4})}  

\newcommand{\HH}{\mathbb{H}}

\newcommand{\bn}{\mathbb{N}}

\newcommand{\calI}{{\mathcal {I}}}
\newcommand{\calJ}{{\mathcal {J}}}

\newcommand{\spa}{\operatorname{span}}

\newcommand{\rank}{\operatorname{rank}}

\newcommand{\M}{(MD)}

\newcommand{\bq}{\begin{equation}}
\newcommand{\eq}{\end{equation}}

\renewcommand{\span}{\mathop{\mathrm{span}}}

\begin{document}

\title{Spanning and independence properties of frame partitions}

\author[B. G. Bodmann]{Bernhard G. Bodmann}

\author[P. G. Casazza]{Peter G. Casazza}

\author[V. I. Paulsen]{Vern I. Paulsen}

\author[D. Speegle]{Darrin Speegle}

\thanks{The first author was supported by
NSF Grant DMS-0807399, the second author by
NSF DMS-0704216, the third author by NSF DMS-0600191 and the fourth author by NSF DMS-0354957}

\begin{abstract}
We answer a number of open problems in frame theory concerning
the decomposition of frames into linearly independent and/or spanning sets.
We prove that  in finite dimensional Hilbert spaces, Parseval frames
with norms bounded away from $1$ can be decomposed into
a number of sets whose complements are spanning, where the number of these sets only depends
on the norm bound. We also prove, assuming the Kadison-Singer conjecture is true, that this
holds for infinite dimensional Hilbert spaces. Further, we prove
a stronger result for
Parseval frames whose norms
are uniformly small, which shows that in addition to
the spanning property, the sets can be chosen to be independent, and
the complement of each set to contain a number of disjoint, spanning sets.
\end{abstract}

\maketitle

\section{Introduction}

A family of vectors $\{f_i\}_{i\in I}$ is a \textit{frame} for a Hilbert space
$\HH$ if there are constants $0<A\le B<\infty$ satisfying 
\[ A\|x\|^2 \le \sum_{i\in I}|\langle x,f_i\rangle|^2 \le B \|x\|^2,\ \ 
\mbox{for all $x\in \HH$}.\]

The theory of frames in Hilbert spaces has applications covering a broad spectrum of problems in pure
mathematics, applied mathematics and engineering \cite{CT}.  Many fundamental
questions in frame theory involve determining the extent to which
frames can be decomposed into subsets which to some extent resemble
bases.
It is known that these problems are generally difficult to resolve.  For example,
work of the second author shows that decomposing frames into subsets
which are Riesz basic sequences
is equivalent to an important open problem
in analysis -- the 1959 Kadison-Singer
Problem \cite{CT}.  

In this note we will answer a number of open problems concerning the
decomposition of frames into linearly independent and/or spanning sets.
The solutions to these problems, for finite dimensional Hilbert spaces, requires some non-trivial
variations of the Rado-Horn Theorem, which is itself rather delicate. 
We prove that some infinite dimensional analogues of these results would have a positive answer
if the Kadison-Singer Problem has a positive answer.  In particular, the fact that the number 
$R$ appearing in Theorem~3.2 can be chosen 
independent of the dimension of the underlying Hilbert space is 
implied by the assumption that the Kadison-Singer Problem has a
positive answer. Indeed, one of the motivations that led to the study of
these consequences of Kadison-Singer was, initially, a quest for a negative answer to the Kadison-Singer Problem.  However, our results verify that these consequences of a positive answer to the Kadison-Singer Problem are, in fact, true.

In Section~2, we derive a result about Parseval frames with norms bounded away from $1$ that is a consequence of the assumption that the Kadison-Singer Problem has a positive answer. We will see that these results involve questions about spanning sets.  In Section~3, we then show that the finite dimensional versions of these results are uniformly true, that is, with constants that do not depend on the dimension, without the need to assume that the Kadison-Singer Problem has a positive answer. 
The results of Section~3 are in some sense refinements of the Rado-Horn theory and rely strongly on earlier refinements of 
Rado-Horn obtained by the second and fourth authors together with Kutyniok. Section~4 is a further development which 
investigates the consequences of having Parseval frames with uniformly small norms.


\section{Kadison-Singer and Spanning Properties for Frame Partitions}

In this section, we begin with a few observations about Parseval frames and prove a result about spanning sets for Parseval frames whose norms are bounded away from $1$, assuming that the Kadison-Singer Problem has a positive answer.

Given a family of vectors $\mathcal F = \{ f_i \}_{i \in S}$ in a Hilbert space $\HH,$ where $S$ is some index set, and a subset $B \subseteq S,$ we write $\mathcal F_B = \{ f_i\}_{i \in B }$
and let $\HH_B$ denote the closed linear span of $\mathcal F_B.$  Recall that if $\{ f_i \}_{i \in S}$ is a Parseval frame for $\HH$ and $P$ is the orthogonal projection onto some closed subspace, then $\{ P f_i \}_{i \in S}$ is a Parseval frame for $P(\HH).$ Also, recall that $\{ f_i \}_{i \in S}$ is a Parseval frame for $\HH$ if and only if the Gram matrix $G= ( \langle f_j, f_i \rangle)_{i,j \in S}$ is the matrix of a projection operator on $\ell^2(S).$

We begin with a few useful observations.

\begin{proposition}\label{2.1} Let $\{ f_i \}_{i \in S}$ be a Parseval frame for $\HH,$ let $P$ be an orthogonal projection onto a closed subspace of $\HH$ and let $I$ denote the identity operator on $\HH.$ Then $G= ( \langle f_j, f_i \rangle)_{i,j,\in S},$ $R = ( \langle Pf_j, Pf_i \rangle)_{i,j\in S}$ and $Q= ( \langle (I-P)f_j, (I-P)f_i \rangle )_{i,j\in S}$ are the matrices of projection operators on $\ell^2(S)$ with $G= R + Q.$
Moreover, $P=I$ if and only if $1$ is not an eigenvalue of $Q.$ 
\end{proposition}
\begin{proof} The equality $G = R+Q$ is immediate from 
$R=( \langle Pf_j, f_i \rangle)$ and $Q= ( \langle (I-P)f_j, f_i \rangle )$
for each $i,j\in S$,
and from the linearity of the inner product in the first entry.
The fact that $G, Q$ and $R$ are matrices of projections follows from the fact that the vectors $\{f_i\}_{i \in S}$
form a Parseval frame.

For the final statement, note that $P=I$ if and only if $Q=0.$ But since $Q$ is a projection, $Q=0$ if and only if $1$ is not an eigenvalue.
\end{proof} 

Given a subset $B \subseteq S,$ we let $D_B=(d_{i,j})_{i,j \in S}$ denote the bounded operator 
on $\ell^2(S)$ whose matrix is the diagonal matrix with $d_{i,i} =1$ when $i \in B$ and $d_{i,j}=0$ when $i \in B^c$ or $j \in B^c$, the complement of the set $B.$

\begin{proposition}\label{2.2}  Let $\{ f_i \}_{i \in S}$ be a Parseval frame for $\HH,$ let $G=(\langle f_j , f_i \rangle)_{i,j \in S}$ denote its Gram matrix,  and let $B \subseteq S.$ Then $\HH_B = \HH$ if and only if $1$ is not an eigenvalue of $D_{B^c}GD_{B^c}.$
\end{proposition}
\begin{proof} Let $P$ denote the projection onto $\HH_B$ and apply Proposition \ref{2.1}. Since for $j \in B, f_j \in \HH_B,$ we have that when $j \in B$ then $\langle f_j, f_i \rangle = \langle P f_j, P f_i \rangle .$  More generally, the matrices $G$ and $R$ are equal in any entry $(i,j)$ provided that $i \in B$ or $j \in B.$ Thus, the matrix $Q$ must be $0$ in any such entry.
Hence we obtain the operator inequalities $0 \le Q = D_{B^c}QD_{B^c} \le D_{B^c}GD_{B^c} \le D_{B^c}.$ 

Now, if $1$ is not an eigenvalue of $D_{B^c}GD_{B^c},$ then these inequalities imply that $1$ is not an eigenvalue of $Q.$ Invoking the preceding proposition, we get $P=I,$ and so $\HH_B = \HH.$

Conversely, assume that $1$ is an eigenvalue of $D_{B^c}GD_{B^c}.$
Write $G=V V^*$ where $V: \HH \to \ell^2(S)$ is the analysis operator
of the Parseval frame. Since $D_{B^c}GD_{B^c} = (V^*D_{B^c})^*(V^*D_{B^c}),$
we have that $(V^*D_{B^c})(V^*D_{B^c})^* = V^* D_{B^c} V$ also has
eigenvalue $1.$ By the Parseval property, $V$ is an isometry, $V^* V = I$, necessarily 
$V^* D_{B} V= I - V^* D_{B^c} V$ has eigenvalue zero. Thus the range of $V^* D_B$ is orthogonal to
the corresponding eigenvectors. Since the closure of the range of $V^* D_B$ is by definition $\HH_B$, it is not equal to $\HH$.
\end{proof}

The last result yields a complementarity principle between spanning and linear independence.

\begin{proposition} Let $\HH$ be a Hilbert space with orthonormal basis $\{ e_j\}_{j \in S },$ let $P$ be the orthogonal projection onto a closed subspace of $\HH,$ and let $B \subseteq S.$  Then the linear span of $\{Pe_j\}_{ j \in B }$ is dense in $P(\HH)$ if and only if the operator  $(\langle (I-P)e_j, (I-P)e_i \rangle)_{i,j \in B^c}$ on $\ell^2(B^c)$ is one-to-one.
\end{proposition}
\begin{proof} Note that the set $\{Pe_j: j \in S \}$ is a Parseval frame for $P(\HH).$ Hence, the  span of $\{Pe_j:j \in B \}$ is dense in $P(\HH)$ if and only if the matrix $Q= (\langle Pe_j, Pe_i \rangle)_{i,j \in B_c}$ does not have 1 as an eigenvalue. But since $I_{\ell^2(B^c)} - Q = (\langle (I-P)e_j, (I-P)e_i \rangle)_{i,j \in B^c},$ $Q$ not having eigenvalue 1 is equivalent to the latter matrix having a trivial kernel.
\end{proof}

\begin{corollary}
If $\HH$ is finite dimensional, then $\{Pe_j\}_{ j \in B }$ spans $P(\HH)$ if and only if the set $\{ (I-P)e_j: j \in B^c \}$ is linearly independent. 
\end{corollary}

For our next result, we will be assuming that the Anderson Paving Problem has an affirmative answer. The Anderson Paving Problem and the Kadison-Singer Problem are known to be equivalent \cite{An}.
There are several equivalent versions of Anderson's Paving Problem. The particular version that we shall use asserts the following:

 For each $0< s <1,$ there exists an $r$ depending only on $s,$ such that if $H=(h_{i,j}) \in B(\ell^2(\bn))$ is any operator with $h_{i,i} =0,$ for every $i,$ then there exists a partition of $\bn$ into $r$ disjoint sets $A_1 \cup \cdots \cup A_r = \bn,$ with $\|D_{A_k}HD_{A_k}\| \le s \|H\|,$ for $k=1,...,r.$

\begin{theorem} Let $0 < \delta < 1.$ If the Anderson Paving Problem has a positive answer, then there exists an $r$ depending on $\delta,$ such that whenever  $\{ f_n \}_{n \in \bn}$ is a Parseval frame for a Hilbert space $\HH,$ with $\|f_n\|^2 \le 1 - \delta $
for all $n \in \bn$,
 then there exists  a partition of $\bn$ into $r$ disjoint sets, $A_1 \cup \cdots \cup A_r = \bn,$ such that $\HH_{A_k^c} = \HH,$ for $k=1,..., r.$
\end{theorem}
\begin{proof}  Let $G= ( \langle f_j, f_i \rangle)$ and let $E(G)$ denote the diagonal part of $G,$ so that $0 \le E(G) \le (1 - \delta) I.$  Since $0 \le G \le I,$ we have that $(\delta -1)I \le -E(G) \le G - E(G) \le G \le I.$  Hence, $H = G - E(G)$ has 0 diagonal and $\|H\| \le 1.$  

Set $s = \delta/2$ in the statement of Anderson's Paving and let $r$ be the corresponding integer. Then we may pick disjoint sets, $A_1 \cup \cdots A_r = \bn$ such that $\|D_{A_k}HD_{A_k}\| \le \delta/2,$ for every $k \in \{1,2, \dots,r\}.$

Hence, we have that $0 \le D_{A_k}GD_{A_k} = D_{A_k}E(G)D_{A_k} + D_{A_k}HD_{A_k} \le (1 - \delta)I + (\delta/2)I = (1 - \delta/2)I.$ Thus, $\|D_{A_k}GD_{A_k}\| < 1$ and it follows that $1$ can not be an eigenvalue. From the preceding proposition, it follows that $\HH_{A_k^c}= \HH,$ for each $k \in \{1,2, \dots,r\}.$
\end{proof}

The study of Parseval frames with norms bounded away from $1$ is in some sense complementary to other results relating Parseval frames and the Kadison-Singer Problem, e.g.~\cite{CCLV}, since most other work relating these problems focuses on Parseval frames whose norms are bounded away from $0$ rather than $1$ and focuses on linear independence rather than spanning.  However, having Parseval frames that are norm-bounded away from $1$ has the added advantage that when one projects onto a subspace, then the projections of these vectors is a Parseval frame for the subspace that is bounded away from $1$ by the same bound. In contrast, a Parseval frame with norms that are bounded away from $0$ might no longer have norms bounded away from $0$ when one projects it onto a subspace. 


\section{Spanning properties for partitions of Parseval frames with norms bounded away from one}

The previous section illustrates that Anderson's paving would provide a partition of certain norm-bounded 
Parseval frames $\{f_i\}_{i\in I}$
into a number of sets with specific spanning properties. Now we show that
the existence of such a number, $r$,  can be obtained independently of
the assumption of Anderson Paving, if the Hilbert space is finite dimensional. 
Moreover, our choice of $r$ depends only on the norm bound $1-\delta,$ and not 
the dimension of the space, and an explicit formula for $r$ as a function of $\delta$ is provided.


Recall that a {\emph {matroid}} is a finite set $X$ together with a collection of subsets of $X$, $\calI$, which satisfy three properties:

\begin{enumerate}
\item $\emptyset\in \calI$
\item if $I_1\in \calI$ and $I_2\subset I_1$, then $I_1\in \calI$, and 
\item if $I_1,I_2\in \calI$ and $|I_1| < |I_2|$, then there exists $x\in I_2 \setminus I_1$ such that $I_1 \cup\{x\} \in \calI$.
\end{enumerate}

We will say that elements of $\calI$ are independent.  We also recall that the rank of a set $E\subset X$ is defined to be the cardinality of a maximal independent set contained in $E$.  

Now, given a set of vectors $\{f_j:j\in J\}$ which spans $\mathbb H_N$, we say $J\in \calJ$ if $\{f_j:j\not\in J\}$ spans $\mathbb H_N$. It is straightforward to verify that $(J, \calJ)$ forms a matroid.  Indeed, the first two properties are immediate and the third property reduces after taking complements to the fact that if $\{f_j : j\in E_1\}$ and $\{f_j:j\in E_2\}$ both span $\mathbb H_N$, and $|E_1| > |E_2|$, then there exists $x\in E_1\setminus E_2$ such that $\{f_j:j\in E_1, j\not= x\}$ spans.

We note here that for a natural number $n$,  rank$(E) \ge n$ if and only if there is a set $F\subset E$ such that $|F| = n$ and $\{f_j:j\not \in F\}$ spans $\mathbb H_N$.  

Finally, we recall the Rado-Horn Theorem \cite{H, R} in the context of matroids.

\begin{theorem}\cite{EF}  Let $(X, \calI)$ be a matroid, and let $R$ be a positive integer.  A set $J\subset X$ can be partitioned into $R$ independent sets if and only if for every subset $E\subset J$,
\begin{equation}\label{eq1}
\frac {|E|}{\rank(E)} \le R.
\end{equation}

\end{theorem}

\begin{theorem} \label{T1} 
Let $\delta > 0$.  Suppose that $\{f_j:j\in J\}$ is a Parseval frame for $\mathbb H_N$ with $\|f_j\|^2 \le 1 - \delta$ for all $j\in J$.  Let $R \in \bn$, $R \ge \frac1\delta$.  Then, it is possible to partition $J$ into $R$ sets $\{A_1,\ldots, A_R\}$ such that for each $1\le r\le R$, the family $\{f_j:j\not\in A_r\}$ spans $\mathbb H_N$.  
\end{theorem}

\begin{proof}  Let $\calJ = \{E\subset J: \spa \{f_j:j\not\in E\} = \mathbb H_N \}$.  Since a Parseval frame must span, we have that $(J, \calJ)$ is a matroid.  By the Rado-Horn Theorem, it suffices to show (\ref{eq1}) for each subset of $J$.  Let $E\subset J$.  Define $S = \spa\{f_j:j\not\in E\}$, and let $P$ be the orthogonal projection onto $S^\perp$.   Since the orthogonal projection of a Parseval frame is again a Parseval frame, we have that $\{Pf_j:j\in J\}$ is a Parseval frame for $S^\perp$.  Moreover, we have
\begin{eqnarray*}
\dim S^\perp &=& \sum_{j\in J} \|Pf_j\|^2 = \sum_{j\in E} \|Pf_j\|^2 \\
&\le& |E|(1 - \delta).
\end{eqnarray*}
Let $M$ be the largest integer smaller than or equal to $|E|(1 - \delta)$.  Since $\dim S^\perp \le M$, we have that there exists a set $E_1 \subset E$ such that $|E_1| = M$ and $\spa \{P f_j:j\in E_1\} = S^\perp$.  Let $E_2 = E\setminus E_1$.  We show $E_2$ is independent.

Write $h \in \mathbb H_N$ as $h = h_1 + h_2$, where $h_1\in S$ and $h_2 \in S^\perp$.   We have that $h_2 = \sum_{j\in E_1} \alpha_j Pf_j$ for some choice of $\{\alpha_j:j\in E_1\}$.  Write $\sum_{j\in E_1} \alpha_j f_j = g_1 + h_2$, where $g_1 \in S$.  Then, there exist $\{\alpha_j:j\not\in E\}$ such that $\sum_{j\not\in E} \alpha_j f_j = h_1 - g_1$.  Then, we have
\[
\sum_{j\not\in E_2} \alpha_j f_j = h,
\]
as desired.

Now, since $E$ contains an independent set of cardinality $|E| - M$, it follows that $\rank(E) \ge |E| - M \ge |E| - |E|(1 - \delta) = \delta |E|$.  Therefore, 
\[
\frac {|E|}{\rank(E)} \le \frac 1\delta \le R,
\]
as desired.  
\end{proof}

We note that it is not possible, in general, to get the partition in Theorem \ref{T1} to
have the property that the $\{f_j\}_{j\in A_i}$ are linearly independent.  Again, the
problem is that we do not have a lower bound on the norms of the frame
vectors and so there can be an arbitrarily large number of them.  That is, there can
be too many frame vectors to be able to partition them into $R$ linearly independent
sets. However, we will see that it is possible to achieve a partition in which
all sets but one are linearly independent and spanning, if the norms of the vectors
are uniformly small.

\section{Spanning and linear independence properties for Parseval frames with uniformly small norms}

In this section we obtain a strengthening of the preceding section with the help of
a generalization of the Rado-Horn Theorem due to Casazza,
Kutyniok and Speegle \cite{CKS}.

\begin{theorem}\label{T2}
Let $\{f_i\}_{i\in I}$ be a finite collection of vectors in a vector space $X$ and let $M\in 
{\mathbb N}$.  The following conditions are equivalent:

(1)  There exists a partition $\{I_j\}_{j=1}^M$ of $I$ so that for each $j$, $\{f_i\}_{i\in I_j}$
is linearly independent.

(2)  For all $J\subset I$,
\[ \frac{|J|}{dim\ span\ \{f_i\}_{i\in J}} \le M.\]

Moreover, in the case that the above conditions fail, there exists a partition 
$\{I_j\}_{j=1}^M$ of $I$ and a subspace $S$ of $X$ such that the following three
conditions hold.
\vskip12pt
\ \ \ \ (a)  For all $1\le j\le M$, $S= span\ \{f_i:i\in I_j,\ \mbox{and}\ f_i\in S\}$.
\vskip12pt
\ \ \ \ (b)  For $J = \{i\in I:f_i\in S\}$, 
\[ \frac{|J|}{dim\ span\ \{f_i\}_{i\in J}}>M.\]
\ \ \ \ \ \ \  (c)  For each $1\le j\le M$, 
\[ \sum_{i\in I_j,f_i\notin S}\alpha_if_i =0,\ \ \mbox{implies}\ \ \alpha_i=0,\ \ 
\mbox{for all i}.\]
In particular, for each $1\le j \le M$, $\{f_i:i\in I_j,\ f_i\notin S\}$ is linearly independent.
\end{theorem}

We also need a slight generalization of  a result of Casazza and Tremain \cite{CT}.

\begin{proposition}\label{P5}
Let $r,k,N$ be  natural numbers with $0< k <N$ and let $\{f_i\}_{i=1}^{rN+k}$ be an 
equal norm Parseval frame for an $N$-dimensional Hilbert space $\mathbb H_N$.
Then $\{f_i\}_{i=1}^{rN+k}$ can be partitioned into $r+1$ linearly independent  sets.
If $k=0$, $\{f_i\}_{i=1}^{rN}$ can be partitioned into $r$ linearly independent spanning
sets.
\end{proposition}

\begin{proof}

Since $\{f_i\}_{i=1}^{rN+k}$ is an equal norm Parseval frame, we have
\[ N = \sum_{i=1}^{rN+k}\|f_i\|^2 = (rN+k)\|f_j\|^2,\ \ \mbox{for all $j=1,2,\ldots,rN+k$}.\]
That is,
\[ \|f_i\|^2 = \frac{N}{rN+k},\ \ \mbox{for all $i=1,2,\ldots,N$}.\]
We will verify that the assumption of the Rado-Horn Theorem holds for $r+1$.  Choose
$J\subset \{1,2,\ldots,rN+k\}$.  Let $P$ be the orthogonal projection of $\mathbb H_N$ onto
span $\{f_i\}_{i\in J}$.  Since $\{Pf_i\}_{i\in J}$ is a Parseval frame for its span we have
\[ \dim \span\ \{f_i\}_{i\in J} = \sum_{i=1}^{rN+k}\|Pf_i\|^2 \ge \sum_{i\in J}\|Pf_i\|^2
= \sum_{i\in J}\|f_i\|^2 = \frac{N|J|}{rN+k}.\]
That is,
\[ \frac{|J|}{\dim \span\ \{f_i\}_{i\in J}}\le \frac{rN+k}{N}.\]
That is,
\[ \frac{|J|}{\dim \span\ \{f_i\}_{i\in J}}\le  \begin{cases} r & \mbox{if $k=0$}\\
 r+1 & \mbox{if $0<k<N$}
 \end{cases} \]

The result now follows by the Rado-Horn Theorem and the fact that in the case $k=0$,
we have partitioned an $rN$ element set into $r$ linearly independent sets
in an $N$-dimensional Hilbert space $\mathbb H_N$, and hence,
each must contain exactly $N$ elements and so it must be a spanning set.  
\end{proof}

We now want to strengthen Proposition \ref{P5} to show that we can actually partition
our family of vectors into a linearly independent set and  $r$ linearly independent spanning sets.

\begin{lemma}\label{lem1}
Let $\{f_i\}_{i\in I_j}$, $j=1,2,\ldots r$ be linearly independent families of vectors in 
an $N$-dimensional Hilbert space $\mathbb H_N$.
Assume there is a partition of $\cup_{j=1}^r I_j$ into $\{A_j\}_{j=1}^r$ so that
\[ \span \ \{f_i\}_{i\in A_j} = \mathbb H_N,\ \ \mbox{for all $j=1,2,\ldots,r$}.\]
Then
\[ \span \  \{f_i\}_{i\in I_j} = \mathbb H_N,\ \ \mbox{for all $j=1,2,\ldots,r$}.\]
\end{lemma}

\begin{proof}
For all $j=1,2,\ldots,r$,
the fact that $\{f_i\}_{i\in I_j}$ are linearly independent implies that the
 dimension of the span of $\{f_i\}_{i\in I_j} = |I_j|.$ Also, the fact that  $\{f_i\}_{i\in A_j}$
span $\HH_N$ implies $|A_j|\ge N$.  Now, we have
\[ Nr \ge \sum_{j=1}^r \dim \span\ \{f_i:i\in I_j\} = \sum_{j=1}^r |I_j| = 
|\cup_{j=1}^r I_j| = |\cup_{j=1}^r A_j| = \sum_{j=1}^r |A_j|  \ge Nr.\]
Hence,
\[ \sum_{j=1}^r \dim \span\ \{f_i:i\in I_j\} = Nr,\]
and so
\[ \dim \span\ \{f_i:i\in I_j\} =N,\ \ \mbox{for every $j=1,2,\ldots,r$}.\]
\end{proof}

Now we can partition frames into spanning sets.  

\begin{proposition}\label{P6}
Let $\{f_i\}_{i\in I}$ be a frame for $\mathbb H_N$ with lower
frame bound $A$ and $\|f_i\|^2 \le 1$ for all $i \in I$.  Let $r = \lfloor A \rfloor$.
  Then there exists a partition  $\{I_j\}_{j=1}^r$ of $I$ so that
\[ \span\ \{f_i:i\in I_j\} = \mathbb H_N,\ \ \mbox{for all $j=1,2,\ldots,r$}.\]
In particular, the number of frame vectors in a unit norm frame
with lower frame bound $A$ is greater than or equal to$\lfloor A\rfloor N$.
\end{proposition}

\begin{proof}
We replace $\{f_i\}_{i\in I}$ by $\{\frac{1}{\sqrt{r}}f_i\}_{i\in I}$ so that our
frame has lower frame bound greater than or equal to $1$ and
\[ \|f_i\|^2 \le \frac{1}{r},\ \ \mbox{for all $i\in I$}.\]
Assume the frame operator for $\{f_i\}_{i\in I}$ has eigenvectors $\{e_j\}_{j=1}^N$
with respective eigenvalues $\lambda_{1}\ge \lambda_{2}\ge \ldots \lambda_N \ge 1$.
We proceed by induction on $N$.
\vskip12pt
\noindent $N=1$:  Since
\begin{equation}\label{E1}
 \sum_{i\in I}\|f_i\|^2 \ge1,\ \ \mbox{and}\ \ \|f_i\|^2 \le \frac{1}{r},
 \end{equation}
it follows that $|\{i\in I:f_i \not= 0\}|\ge r$ and so we have a partition into
$r$ spanning sets.
\vskip12pt
\noindent Assume the inductive hypothesis holds for $\mathbb H_N$ and consider $\mathbb H_{N+1}$. 
\vskip12pt
We check two cases:
\vskip12pt
\noindent {\bf Case I}:  Suppose there exists a partition $\{I_j\}_{j=1}^r$ of $I$ so that
$\{f_i\}_{i\in I_j}$ is linearly independent for all $j=1,2,\ldots,r$.
\vskip12pt
In this case, 
\[ N+1 \le (N+1)\lambda_N \le \sum_{j=1}^{N+1}\lambda_j= \sum_{i\in I}\|f_i\|^2 \le |I| \frac{1}{r},\]
and hence,
\[ |I|\ge r(N+1).\]
However, by linear independence, we have
\[ |I|= \sum_{j=1}^r|I_j| \le r(N+1) .\]
Thus, $|I_j|=N+1$ for every $j=1,2,\ldots,r$ and so $\{f_i\}_{i\in I_j}$ 
are all spanning.
\vskip12pt
\noindent {\bf Case II}:  Our family cannot be partitioned into $r$ linearly independent
(spanning) sets.
\vskip12pt
In this case, let $\{I_j\}_{j=1}^r$ and a subspace $\emptyset \not= S \subset \mathbb H_{N+1}$ be
given by Theorem \ref{T2}.  If $S = \mathbb H_{N+1}$, we are done.  Otherwise, let $P$ be the orthogonal
projection onto the subspace $S$.  
Let
\[ I_j' = \{i\in I_j:f_i\notin S\},\ \ I' = \cup_{j=1}^r I_j'.\]
Theorem \ref{T2} (c) implies that $\{f_i\}_{i\in I_j'}$ is linearly independent for all
$j=1,2,\ldots,r$. To see this,
note that the non-zero elements of $\{(I-P)f_i\}_{i\in I}$ are $\{(I-P)f_i\}_{i\in I'}$.
Fix $1\le j\le r$ and assume there are scalars $\{\alpha_i\}_{i\in I_j'}$
with
\[ \sum_{i\in I_j'}\alpha_i(I-P) f_i = 0.\]
This implies {by} Theorem \ref{T2} (c): 
\[ \sum_{i\in I_j'}\alpha_if_i \in S,\ \ \mbox{and so}\ \ \alpha_i =0,\ \mbox{for all $i\in I_j'$}.\]

Now, $\{(I-P)f_i\}_{i\in I'}$
has lower frame bound 1 in $(I-P)(\mathbb H_{N+1})$, dim $(I-P)(\mathbb H_{N+1})\le N$ and
\[ \|(I-P)f_i\|^2 \le \|f_i\|^2 \le \frac{1}{r},\ \ \mbox{for all $i\in I'$}.\]
Applying the induction hypothesis, we can find a partition $\{A_j\}_{j=1}^r$ of $I'$ with
\[ \span \ \{(I-P)f_i\}_{i\in A_j} = (I-P)(\mathbb H_{N+1}),\ \ \mbox{for all $j=1,2,\ldots,r$}.\]
Now, we can apply Lemma \ref{lem1} together with the partition $\{A_j\}_{j=1}^r$ to conclude:
\[ \span\ \{(I-P)f_i\}_{i\in I_j'}= (I-P)(\mathbb H_{N+1}),\]
and hence 
\[ \span\ \{f_i\}_{i\in I_j} = \span\ \{S,\{(I-P)f_i\}_{i\in I_j'}\} = \mathbb H_{N+1}.\]
\end{proof}

Note that we cannot expect to get any linear independence in Proposition \ref{P6}
because our vectors can have arbitrarily small norms and hence there can be
an arbitrarily large number of them.  However, we can remove appropriate 
vectors from the last $r-1$-sets until they are linearly independent and
spanning.  Putting the removed vectors into the first set, we get a partition
into a spanning set and $r-1$ linearly independent spanning sets.

\begin{corollary}\label{P6.5}
Let $\{f_i\}_{i\in I}$ be a Parseval frame for $\mathbb H_N$ and $r$
a natural number so that $\|f_i\|^2 \le \frac{1}{r}$
for every $i\in I$.  Then there is a partition  $\{I_j\}_{j=1}^r$ of $I$ so that
\[ \span\ \{f_i:i\in I_j\} = \mathbb H_N,\ \ \mbox{for all $j=1,2,\ldots,r$}.\]
\end{corollary}

In the following we answer a question concerning the partition of equal norm Parseval
frames into spanning sets which continues Proposition \ref{P5}, and had been
left open in \cite{CT}.  

We will be considering partitions which maximize dimensions in a very particular way.  

\begin{definition}
Let $\{f_i\}_{i \in I}$ be a family of vectors.
We say that a partition $\{I_1,\ldots,I_M\}$ of the index set $I$ has the maximality property \M\ if whenever 
$\{J_i\}_{i=1}^M$ is any partition of $I$ satisfying that for all $1\le i\le M$,
dim span $\{f_j\}_{j\in J_i} \ge$ dim span $\{f_j\}_{j\in I_i}$, then dim span $\{f_j\}_{j\in I_i}
=$ dim span $\{f_j\}_{j\in J_i}$ for all $i=1,2,\ldots,M$.
\end{definition}

A straightforward consequence of maximality  is the following:

\begin{lemma}\label{keylemma} Let $\mathcal F = \{f_i: i\in I\}$ be a finite collection of vectors in a 
vector space.  Let $M\in \mathbb N$ and $\{I_j:
j=1,\ldots,M\}$ be a partition of
$I$ satisfying property \M.  
If $f_k\in I_p$ and $f_k = \sum_{l \in I_p, l\not=
 k} \alpha_l f_l$, then $f_k\in \span(\FF_{I_j})$ for all $1\le j \le M$.
\end{lemma}

\begin{proof}  Assuming the hypothesis of the lemma, if $f_k =  \sum_{l \in I_p, l\not=
 k} \alpha_l f_l$, then removing $f_k$ from $I_p$ keeps $\dim \span (\FF_{I_p})$ constant. By property \M, moving $f_k$ into another $I_j$, $j\not= p$ cannot increase $\dim \span(\FF_{I_j})$, and the result follows.
\end{proof}


If there are linear dependent sets in a partition having property \M\ then we can move suitable vectors from one set to
another.  The following definition will be used to help us keep track of which vectors are being moved.   

\begin{definition} Let $\{f_i: i\in I\}$ be a collection of vectors in a vector space and 
let $\{I_j: j = 1,\ldots,M\}$ be a
partition of
$I$.  
We define a \emph{chain of length one} to be a set $\{(a,b)\}$
with $a \in I_b$, $b \in \{1, 2, \dots, M\}$ and $f_{a} =  \sum_{j\in I_{b}, j \ne a} \alpha_j f_j$
for some choice of constants $\{\alpha_j\}_{j \in I_b, j \ne a}$.
We define a \emph{chain of length $n$} to be a finite
sequence
$\{(a_1, b_1),\ldots, (a_n, b_n)\}$, where $a_i
\in I$ and $b_i\in \{1,\ldots,M\}$, such that 
\begin{itemize}
\item  $(a_1,b_1)$ is a chain of length one,
\item  for $2\le i \le n$, $a_i \in I_{b_i}$ and $f_{a_i} = \alpha f_{a_{i-1}} + \sum_{j\in I_{b_i}, j\not= a_i} \alpha_j f_j$ for
some
$\alpha \not= 0$, and
\item  $a_i \not= a_k$ for $i \not= k$.
\end{itemize}
A chain of length $n$ starting with $a_1 \in L\subset I$ and ending at $a_n\in I$ is a \emph{chain of minimal length starting in $L$ and ending at $a_n$} if every chain
starting in $L$ and ending at $a_n$ has length greater than or equal to $n$. 
\end{definition}

We recall the following lemma. 
 
\begin{lemma}[Casazza, Kutyniok, Speegle]\label{chainlem} Let $\{f_i: i\in I\}$ be a collection of vectors in a vector space, 
let $\{I_j: j = 1,\ldots,M\}$ be a
partition of
$I$, and let $L \subset I_1$.

If $\{(a_1,b_1),\ldots,(a_n, b_n)\}$ is a chain of minimal length starting in $L$ and ending at $a_n$, then for each
$1\le i
\le n$,
$\{(a_1,b_1),\ldots,(a_i,b_i)\}$ is a chain of minimal length starting in $L$ and ending at $a_i$.
\end{lemma}

\begin{proof} By induction 
it suffices to show that $\{(a_1,b_1),\ldots,(a_{n-1}, b_{n-1})\}$ is a chain of minimal length.  
Suppose, for
the sake of contradiction, that there did exist a chain $\{(u_1,v_1), \ldots, (u_k,v_k)\}$ such that $u_k = a_{n-1}$ and $k < n- 1$.  Since $\{\chain ab1n\}$ is a
chain, 
$$
f_{a_n} = \alpha f_{a_{n-1}} + \sum_{j\in I_{b_{n}}, j\not= a_n} \alpha_j f_j 
$$
for some $\alpha \not=0$.    Therefore, either $\{\chain uv1k, (a_n, b_n)\}$ is a chain with length $k+1<n$ or $a_n = u_i$ for some $i\le k $, either of which contradicts the minimality of $n$.
\end{proof}

\begin{lemma} 
Let $\mathcal F = \{f_i: i\in I\}$ be a finite collection of vectors, and $M\in \NN$. There exists among all the partitions of $I$ into
$M$ non-empty subsets a partition $\{I_1, I_2, \dots I_M\}$ with the property \M. This partition can be chosen 
so that $\FF_{I_j}$ is linearly independent for all $2 \le j \le M$.
\end{lemma} 
\begin{proof}
The set of partitions of $I$ into $M$ sets has a partial ordering with respect to which two partitions $\{I_j\}_{j=1}^M$ 
and $\{J_j\}_{j=1}^M$ satisfy  $\{I_j\}_{j=1}^M \le \{J_j\}_{j=1}^M$ if $\dim \FF_{I_j} \ge \dim \FF_{J_j}$ for all
$j \in \{1, 2, \dots, M\}$. $I$ is a finite set, so there are maximal elements. By definition, these partitions have the property \M. 
 
Assume that there is a partition with property \M\ which contains more than one set for which the associated vectors are linearly dependent, say $I_1$ and $I_2$. We can then successively remove indices from $I_2$ and place them into $I_1$ if the associated
vectors are linear combinations of others remaining in the set indexed by $I_2$. After finitely many such moves, $\FF_{I_2}$
is linearly independent. Moreover, by Lemma~\ref{keylemma}, the span of $\FF_{I_1}$ and $\FF_{I_2}$ retain
their dimensions, which means the maximality is preserved.
\end{proof}

If $\FF_{I_2},\ldots, \FF_{I_M}$ are linearly independent, $L = \{i\in I_1: f_i =
\sum_ {j\in I_1, j\not= i} \alpha_j f_j\}$, and $\{\chain ab1n\}$ is a  chain of minimal length starting in $L$, it follows that for each $1\le i < n$, $b_i \not= b_{i + 1}$.  In this case, we can track the changes in the partition as vectors are moved among the sets
in a straightforward manner.

\begin{definition}
If $\FF_{I_2},\ldots, \FF_{I_M}$ are linearly independent, then proceeding by induction, we can define
\[
U_k^1 = I_k, \,\,\,\,\, 1\le k \le M,
\]
and for $2\le i\le n$,
\begin{eqnarray*}
U_k^i &=& U_k^{i-1} \,\text{ for } k\not= b_{i-1}, k\not= b_i,\\
U_{b_i}^i &=& U_{b_i}^{i-1} \cup \{a_{i-1}\},\\
U_{b_{i-1}}^i &=& U_{b_{i-1}}^{i-1} \setminus \{a_{i-1}\}.
\end{eqnarray*}
\end{definition}

\begin{lemma} \label{subc2} Let $\mathcal F = \{f_i: i\in I\}$ be a finite collection of vectors, and $\{I_1, I_2, \dots I_M\}$
a partition with the property \M\ for which  $\FF_{I_2}, \dots, \FF_{I_M}$ are linearly independent. Let $L$ be as above
and assume that $\{\chain ab1n\}$ is a minimal chain starting in $L$.
For each $1\le i \le n$, $f_{a_i}$ can then be written as the sum
\begin{equation}\label{4.5}
f_{a_i} = \sum_{j\in I_{b_i}, j\not\in \{a_p: 1\le p \le n\}} \alpha_j f_j + \sum_{j\in U_{b_i}^i \cap \{a_p: 1 \le p < i\}} \alpha_j f_j.
\end{equation}
\end{lemma}

\begin{proof} For the case $i = 1$, note that $a_1\in L$ implies that $f_{a_1} = \sum_{j\in L, j\not= a_1} \alpha_j f_j$ for some choice of $\alpha_j$. 
By Lemma~\ref{chainlem} none of these $j\in L$ can be in $\{a_p: 1\le p \le n\}$ since this would not be a chain of minimal length
starting in $L$.  Recalling that $b_i =
1$, the claim is proven for $i = 1$.  

Proceeding by induction, let $i\in \{1,\ldots,n\}$ and we assume (\ref{4.5}) is true for $1\le k < i$. We will show that it is also true for $i$.  Note that
\begin{eqnarray}
f_{a_i} &=& \alpha f_{a_{i-1}} + \sum_{j\in I_{b_i}, j\not= a_{i}} \alpha_j f_j\\
 &=& \alpha f_{a_{i-1}} + \sum_{j\in I_{b_i} \cap U_{b_i}^i, j \not= a_i} \alpha_j f_j + \sum_{j\in I_{b_i} \setminus U_{b_i}^i} \alpha_j f_j\nonumber\\
\label{eqqq} &=& \alpha f_{a_{i-1}} + \sum_{j\in I_{b_i} \cap U_{b_i}^i, j\not = a_i} \alpha_j f_j + \sum_{j\in I_{b_i} \cap \{a_p: 1\le p < i - 1\}}
\alpha_j f_j,
\end{eqnarray}
where we have used in the last two lines that $I_{b_i} \cap \{a_p: 1\le  p < i-1\} = I_{b_i} \setminus U_{b_i}^i.$
Now, suppose for the sake of contradiction that there is a $j\in I_{b_i} \cap U_{b_i}^i$ such that $\alpha_j \not= 0$ and $j = a_p$ for some $p > i$. 
Then $\{\chain ab1{i-1}, (a_p,b_i)\}$ is a chain starting in $L$, which contradicts the minimality of the chain $\{\chain ab1n\}$.  So, using the induction hypothesis on each term in the last
sum in (\ref{eqqq}) and combining terms, one obtains 
\[
f_{a_i} = \alpha f_{a_{i-1}} + \sum_{j\in I_{b_i}, j\not\in \{a_p: 1\le p \le n\}} \tilde\alpha_j f_j + \sum_{j\in U_{b_i}^i \cap \{a_p : 1\le p < i\} }\tilde \alpha_j
f_j
\]
with an apporpriate choice of $\tilde\alpha_j$'s.
\end{proof}

\begin{lemma}\label{claim1}  Let $\mathcal F = \{f_i: i\in I\}$ be a finite collection of vectors, and $\{I_j\}_{j =1}^M$ a partition
of the index set $I$ into $M\in \mathbb N$ non-empty sets which has the proeprty \M\ and for which sets $I_2$, $I_3$, \dots $I_{M}$ index linearly independent sets.  Moreover, let $L = \{i \in I_1: f_i = \sum_{j\in I_1, j\not= i} \alpha_j f_j\}$, $L_0 = \{i\in I: $ there is a chain starting in $L$ and ending at $i\}$, and $L_j = L_0 \cap I_j$ for $1\le j \le M$.  If $\{(a_1, b_1),\ldots, (a_n, b_n)\}$ is a  chain of minimal length starting in $L$ and ending at $a_n$, then $f_{a_n} \in \span
(\FF_{L_m})$ for all
$1\le m
\le M$.  
\end{lemma}

\begin{proof}   We show that, if $\{\chain ab1n\}$ is a  chain of minimal length starting in $L$ and ending at $a_n$, then
$f_{a_n}\in
\span(\FF_{L_m})$ for each $1\le m\le M$.  

For $n = 1$, fix $m\in \{1,\ldots,M\},$ and observe that $a_1\in L$.  Hence, by 
Lemma~\ref{keylemma}, we can 
write $f_{a_1} = \sum_{l\in I_m} \alpha_l f_l$.  For each $l$ such that
$\alpha_l\not= 0$, $(a_1,1), (l,m)$ is a chain ending at $l$.  Therefore, 
$f_{a_1} \in \span(\FF_{L_m})$, as desired.

By Lemma \ref{subc2} and the fact that $I_{b_i} \setminus \{a_p: 1\le p \le n\}\subset U^k_{b_i}$ for all $1\le k \le n$, we have that
 $f_{a_i} \in \span (\FF_{U_{b_i}^i \setminus \{a_i\}})$.  Therefore, $\dim\span(\FF_{U_{b_i}^i}) = \dim\span(\FF_{U_{b_i}^{i+ 1}})$.  In particular,  the partition $\{U_k^i:1\le k \le M\}$ satisfies property \M.

By property \M,
Lemma~\ref{subc2}, and Lemma~\ref{keylemma}, $f_{a_n} \in \span(\FF_{U_m^n})$ for each $1\le m \le M$.  Therefore, for $m \not = b_n$, there exist $\alpha_j^{(0)}$
such that
\begin{eqnarray}
f_{a_n} &=& \sum_{j\in U_m^n} \alpha_j^{(0)} f_j = \sum_{j\in U_m^n \cap I_m} \alpha_j^{(0)} f_j + \sum_{j\in U_m^n \setminus I_m} \alpha_j^{(0)} f_j \nonumber \\
&=& \label{ppp} \sum_{j\in U_m^n\cap I_m} \alpha_j^{(0)} f_j + \sum_{j\in \{a_p : b_{p + 1} = m, 1 \le p < n-1\}} \alpha_j^{(0)} f_j.
\end{eqnarray}
By definition of a chain, for each $a_p$ such that $b_{p + 1} = m$ and $1\le p < n - 1$, 
\begin{equation}\label{iii}
f_{a_p} = \alpha^p f_{a_{p + 1}} + \sum_{j\in I_m, j\not= a_{p + 1}} \alpha_j^{(p)} f_j,
\end{equation}
for some choice of $\alpha_j^{(p)}$ and some $\alpha^{(p)}\not= 0$.

Fix $j_0$ such that $\alpha_{j_0}^{(0)} \not= 0$ in (\ref{ppp}).  We show that $j_0\in L_m$, which finishes the proof of
the lemma.  Clearly, if
$j_0\in \{a_1,\ldots, a_n\}$, then we are done, so we assume that $j_0 \not\in \{a_1,\ldots, a_n\}$.  
 
Case 1: There is some $1\le p < n - 1$ such that $b_{p + 1} = m$ and $\alpha_{j_0}^{(p)} \not= 0$.  Then, one can solve (\ref{iii}) for $f_{j_0}$ to obtain 
\[
f_{j_0} = \beta f_{a_p} + \sum_{j\in I_m, j\not= j_0, j\not= a_p} \beta_j f_j
\]
for some $\beta \not= 0$.  Hence, $\chain ab1p, (j_0, m)$ is a chain and $j_0 \in L_m$.  

Case 2: For each $1\le p< n-1$ such that $b_{p + 1} = m$, we have $\alpha_{j_0}^{(p)} = 0$.  We have   
\begin{eqnarray*}
f_{a_n} &=& \sum_{j\in U_m^n \cap I_m} \alpha_j^{(0)} f_j + \sum_{j\in \{a_p : b_{p + 1} = m, 1\le p < n - 1\}} \alpha_j^{(0)} f_j  \\
\label{belangrik} &=&\sum_{j\in U_m^n \cap I_m}  \alpha_j^{(0)} f_j + \sum_{p\in \{p : b_{p+1} = m, 1\le p < n - 1\}} \alpha_{a_p}^{(0)} f_{a_p}\\
&=&\sum_{j\in U_m^n \cap I_m}  \alpha_j^0 f_j + \sum_{p\in \{p : b_{p+1} = m, 1\le p < n - 1\}} \alpha_{a_p}^{(0)} \bigl( \alpha^{(p)} f_{a_{p + 1}} + \sum_{j\in
I_m, j\not= a_{p + 1}} \alpha_j^{(p)} f_{j}\bigr) \\
&=&\alpha_{j_0}^{(0)} f_{j_0} + \sum_{j\in I_m, j\not= j_0} \tilde \alpha_j f_j,
\end{eqnarray*}  
where the first equality is (\ref{ppp}), the second equality is a re-indexing, the third equality follows from (\ref{iii}), and the last equality holds for some choice of $\tilde \alpha_j$ by combining sums, since $\alpha_{j_0}^{(p)} = 0$ for all $1\le p < n-
1$ such that $b_{p + 1} =
m$, and $j_0\not\in \{a_1,\ldots,a_n\}$.  Therefore, 
$\{\chain ab1{n}, (j_0, m)\}$ is a chain and $j_0\in L_m$. 
\end{proof}

The purpose of this is to prove the following:




\begin{theorem}\label{cor5}
Let $\{f_i\}_{i\in I}$ be a finite collection of vectors in a finite dimensional vector space $X$.  Assume

(1)  $\{f_i\}_{i\in I}$ can be partitioned into $r+1$-linearly independent sets, and

(2)  $\{f_i\}_{i\in I}$ can be partitioned into a set and $r$ linearly independent spanning
sets.

Then there is a partition $\{I_i\}_{i=1}^{r+1}$ so that $\{f_j\}_{j\in I_i}$ is a 
linearly independent spanning set for all $i=2,3,\ldots,r+1$ and $\{f_i\}_{i\in I_{1}}$
is a linearly independent set.
\end{theorem}

\begin{proof}
We choose the partition $\{I_i\}_{i=1}^{r+1}$ 
of $I$ that maximizes dim span $\{f_j\}_{j\in I_{1}}$ taken
over all partitions so that the last $r$ sets span $X$.   If $\{J_i\}_{i = 1}^{r + 1}$ is a partition of $I$ such that  for all $1\le i\le r + 1$,
dim span $\{f_j\}_{j\in J_i} \ge$ dim span $\{f_j\}_{j\in I_i}$, then dim span $\{f_j\}_{j\in I_i}
=$ dim span $\{f_j\}_{j\in J_i}$ for all $i=2,\ldots,r+1$ since dim span $\{f_j\}_{j\in I_i} = \dim X$, and $\dim \span \{f_j\}_{j\in I_{1}} = \dim \span \{f_j\}_{j\in J_{1}}$ by construction.
This means, the chosen partition has the property \M\ and
the properties asserted by Lemma~\ref{claim1}.   Suppose that this does not partition $\FF_I$ into linearly independent sets, i.e. $\FF_{I_1}$ is not linearly independent.  As in Lemma~\ref{claim1}, let $L = \{i\in I_1: f_i =
\sum_ {j\in I_1, j\not= i} \alpha_j f_j\}$ be the index set of the ``linearly dependent vectors" in $I_1$, 
$L_0 = \{i\in I: \text {\rm there is a chain starting in $L$ ending at $i$}\},$ and
$L_j = L_0\cap I_j, 1\le j \le r+1.$

Let $S = \span (\FF_{L_0})$.  By Lemma~\ref{claim1}, $S = \span(\FF_{L_j})$ for all $1\le j\le r+1$.  Moreover,
for
$1\le j
\le r+1$, $i\in L_j$ implies that $i\in I_j$ and $f_i \in S$.  Therefore,
\begin{equation*}
S \subset \span \{f_i: i\in L_j\} \subset \span \{f_i: i\in I_j, f_i\in S\} = S.
\end{equation*}

Let $J = \{i\in I : f_i\in S\}$. By
construction, $L\subset J$. Let $d = \dim(S)$ and see that, by the preceding portion of this proof, $\dim \span(\FF_J) = d$.  Moreover, 
\[|J| = |L_1| + \cdots + |L_M| = |L_1| + rd > d(r+1),\] 
because $L_1$ is linearly dependent, since it contains $L$ by virtue of chains of length one.  Therefore, for $J = \{i\in I: f_i\in S\}$,  $\frac {|J|} {\dim\span(\FF_J)} > r+1$. This is in contradiction with assumption (1), which implies
by the Rado Horn theorem that $|J|/d \le r+1$.
%

%
\end{proof}

Combining Propositions \ref{P5} and \ref{P6} with Theorem~\ref{cor5} we conclude:

\begin{corollary}
Let $\{f_i\}_{i\in I}$ be an equal norm Parseval frame for $\mathbb H_N$ with
$|I|=rN+k$ with $0\le k <N$.  Then there is a partition $\{I_i\}_{i=1}^{r+1}$ of $I$
so that for $i\in \{2,\ldots,r+1\}$, $\{f_j\}_{j\in I_i}$ is a linearly independent spanning set
and $\{f_j\}_{j\in I_{1}}$ is linearly independent.
\end{corollary}

If $r\ge 2$ then this result implies that each set of frame vectors has a complement
which is spanning, which was already obtained in Section~3. The insight of this last corollary
is that with the lower norm bound implicit in the equal-norm Parseval property,
the partition can be chosen to consist of linearly independent sets. Moreover, the complement of each set can then be partitioned into at least $r-1$ spanning sets.

\paragraph{\bf Acknowledgments.} The authors would like to thank the American Institute of Mathematics for hospitality and the support of the Structured Quartet Research Ensemble
on the Kadison Singer Problem.


\end{document}